\documentclass[11pt,a4paper]{article}
\usepackage[T1]{fontenc}
\usepackage{cite}
\usepackage{amsmath}
\usepackage{amsthm}
\usepackage{amssymb}
\usepackage{amscd}
\usepackage{graphicx}
\usepackage{epsfig}
\usepackage{color}
\usepackage{mathrsfs}
\usepackage{mathtools}
\usepackage{amsfonts}
\usepackage{stmaryrd}
\usepackage{hyperref}

\renewcommand{\mod}{\operatorname{mod}}

\newcommand{\Class}{\mathscr{C}}
\newcommand{\Cinfty}{\Class^\infty}

\usepackage[left=2.5cm,right=2.5cm,top=3cm,bottom=3.5cm]{geometry}

\newtheorem{theorem}{Theorem}[section]
\newtheorem{proposition}[theorem]{Proposition}

\newtheorem{corollary}[theorem]{Corollary}

\theoremstyle{definition}
\newtheorem{example}[theorem]{Example}

\newtheorem{remark}[theorem]{Remark}
\newtheorem{definition}[theorem]{Definition}

\newcommand*{\NN}{\mathbb{N}}

\newcommand*{\ZZ}{\mathbb{Z}}

\newcommand*{\RR}{\mathbb{R}}
\newcommand*{\R}{\mathbb{R}}

\newcommand{\MM}{\mathcal{M}} 
\newcommand{\CM}{\Class_\MM}

\newcommand{\be}{\begin{equation}}
\newcommand{\ee}{\end{equation}}

\newcommand{\bea}{\begin{eqnarray}}
\newcommand{\eea}{\end{eqnarray}}
\newcommand{\beas}{\begin{eqnarray*}}
\newcommand{\eeas}{\end{eqnarray*}}

\newcommand{\mn}{{\medskip\noindent}}

\newcommand{\no}{{\noindent}}

\newcommand\restr[2]{{
  \left.\kern-\nulldelimiterspace 
  #1 
  \right|_{#2} 
}}
\newcommand{\T}{{\mathsf T}} 
\newcommand{\sT}{{\mathsf T}}

\let\id\Id

\newcommand*{\liedv}[1]{\mathcal{L}_{#1}}

\newcommand{\n}{\nabla}
\newcommand{\pa}{\partial}

\def\End{\mathsf{End}}

\def\la{\langle}
\def\ran{\rangle}

\def\xd{\mathrm{d}}

\newcommand{\za}{\alpha}
\newcommand{\zb}{\beta}
\newcommand{\zf}{\varphi}

\newcommand{\zi}{\iota}
\newcommand{\zw}{\omega}
\newcommand{\ti}{\times}
\def\wa{\operatorname{w}}
\newcommand{\Exp}{\operatorname{Exp}}

\def\bl{\big( }
\def\br{\big) }
\def\Bl{\Big( }
\def\Br{\Big) }
\newcommand{\nm}[1]{\ensuremath{\Vert #1 \Vert}}
\mathchardef\za="710B  
\mathchardef\zb="710C  
\mathchardef\zg="710D  
\mathchardef\zd="710E  
\mathchardef\zve="710F 
\mathchardef\zz="7110  
\mathchardef\zh="7111  
\mathchardef\zvy="7112 
\mathchardef\zi="7113  
\mathchardef\zk="7114  
\mathchardef\zl="7115  
\mathchardef\zm="7116  
\mathchardef\zn="7117  
\mathchardef\zx="7118  
\mathchardef\zp="7119  
\mathchardef\zr="711A  
\mathchardef\zs="711B  
\mathchardef\zt="711C  
\mathchardef\zu="711D  
\mathchardef\zvf="711E 
\mathchardef\zq="711F  
\mathchardef\zc="7120  
\mathchardef\zw="7121  
\mathchardef\ze="7122  
\mathchardef\zy="7123  
\mathchardef\zf="7124  
\mathchardef\zvr="7125 
\mathchardef\zvs="7126 
\mathchardef\zf="7127  
\mathchardef\zG="7000  
\mathchardef\zD="7001  
\mathchardef\zY="7002  
\mathchardef\zL="7003  
\mathchardef\zX="7004  
\mathchardef\zP="7005  
\mathchardef\zS="7006  
\mathchardef\zU="7007  
\mathchardef\zF="7008  
\mathchardef\zW="700A  

\begin{document}
\title{\textbf{Homogeneity actions, N-manifolds, \\ and the Frobenius theorem}
\footnote{The research of JG was funded by the National Science Centre (Poland) within the project WEAVE-UNISONO, No. 2023/05/Y/ST1/00043.}
}
\date{\today}
\author{\\  \Large Janusz Grabowski
        \\ \Large Asier L\'opez-Gord\'on\\ \\
          {\it Institute of Mathematics}\\
                {\it Polish Academy of Sciences}
}
\maketitle
\begin{abstract}\noindent
This paper is devoted to $\NN$-graded supermanifolds $\MM$ whose grading is induced by a homogeneity action, i.e., a smooth action $\R\ni t\mapsto h_t$ of the multiplicative monoid of real numbers on $\MM$. We show that the map $h_0$ is a smooth retraction onto a submanifold $M=h_0(\MM)$, and that $h_0:\MM\to M$ is a fiber bundle with typical fiber $\R^{m|n}$. Using this homogeneity approach, we obtain a simple proof of a homogeneous version of the Frobenius theorem. If, in addition, $h_{-1}$ acts as the parity operator on $\MM$, we provide a geometric characterization equivalent to the recent definition of $\NN$-manifolds due to Bursztyn, Cueca, and Mehta in terms of sheaves of graded algebras with prescribed local models. As a consequence, the homogeneous Frobenius theorem for $\NN$-manifolds proved by these authors appears as a special case of our more general result.

\par

\end{abstract}

\section{Introduction}
The present paper is devoted to $\NN$-graded manifolds, also known as \emph{N-manifolds}, that is, supermanifolds admitting local coordinates of degrees $0,1,\ldots,n$ which commute according to the graded sign rule.

\mn Roytenberg in \cite{Roytenberg:2002} defined an N-manifold as an $\NN$-graded supermanifold whose underlying $\ZZ_2$-grading is induced by the $\NN$-grading modulo $2$, so that coordinates of even (respectively, odd) weight are even (respectively, odd). On the other hand, in \cite{Severa:2005}, \v{S}evera states that ``an N-manifold (shorthand for `non-negatively graded manifold') is a supermanifold with an action of the multiplicative semigroup $(\RR,\cdot)$ such that $-1$ acts as the parity operator.'' Although the equivalence of these two descriptions is generally regarded as folklore, to the best of our knowledge, no complete proof has appeared in the literature. However, it is worth noting that supermanifolds equipped with smooth actions of the monoid $(\RR,\cdot)$ were already studied by J\'o\'zwikowski and Rotkiewicz \cite{Jozwikowski:2016}.

Recently, Bursztyn, Cueca, and Mehta \cite{B.C.M2025} proposed another definition of $\NN$-graded manifolds and established a Frobenius theorem for graded involutive distributions. In their approach, an \emph{$\NN$-manifold} consists of a smooth manifold $M$ endowed with a sheaf of graded algebras locally isomorphic to
\[
\Cinfty(\RR)\otimes \operatorname{S}^\bullet V,
\]
where $\operatorname{S}^\bullet V$ is the graded symmetric algebra of a graded vector superspace
\[
V=V_1\oplus\cdots\oplus V_n,
\]
where the parity of $V_k$ is the parity of $k$. They prove that these objects are equivalent to admissible coalgebra bundles and use this equivalence to establish a Frobenius theorem asserting that an involutive graded distribution is locally generated by homogeneous coordinate vector fields. Their approach relies on a sophisticated algebraic framework involving admissible coalgebra bundles and graded sheaves.

\mn In the present paper, we define an N-manifold as a supermanifold $\MM$ ($\ZZ_2$-graded manifold, see \cite{Kostant:1977,Leites:1980}) endowed with a \emph{homogeneity action}, that is, a smooth action
\[
h:\RR\times\MM\longrightarrow\MM
\]
of the multiplicative monoid $(\RR,\cdot)$ satisfying two additional conditions: $h_0:\MM\to|\MM|$ is the projection onto the body, and $h_{-1}$ coincides with the parity operator. We prove that $\MM$ admits an atlas of local trivializations
\[
h_0^{-1}(U)\simeq U\times\RR^{\alpha|\beta},
\]
equipped with homogeneous coordinates $(x^a,u^i)$ satisfying
\[
h_t(x^a)=x^a,\qquad
h_t(u^i)=t^{w_i}u^i,\qquad w_i\in\NN.
\]
Here $\RR^{\alpha|\beta}\simeq\RR^\alpha\otimes\bigwedge^\bullet\RR^\beta$ denotes the standard superspace. This immediately yields the equivalence of Roytenberg's and \v{S}evera's definitions. We further establish the equivalence with the definition of Bursztyn, Cueca, and Mehta.

More generally, we show that many constructions of classical differential geometry and supergeometry admit natural homogeneous counterparts. Once the relevant geometric objects are required to be compatible with the homogeneity action, classical arguments extend almost verbatim to the graded setting. In particular, by incorporating suitable homogeneity assumptions into the standard proof of the Frobenius theorem for supermanifolds, we obtain a Frobenius theorem for involutive homogeneous distributions. As a consequence, the main theorem of \cite{B.C.M2025} is recovered as a special case by a direct differential-geometric argument.

In the literature, a supermanifold $\mathcal{M}$ is defined either as a set endowed with compatible smooth charts of (super)coordinates taking values in the exterior algebra $\bigwedge\nolimits^\bullet(V)$, or as an ordinary manifold $|\MM|$ equipped with a sheaf of functions which is locally isomorphic to $\Cinfty(U)\otimes \bigwedge\nolimits^\bullet(V)$ \cite{Carmeli:2011,D.M1999,Manin1997}.
Following Rogers \cite{Rogers2007}, we shall refer to the former as \emph{concrete} supermanifolds, and to the latter as \emph{algebro-geometric} supermanifolds. Both definitions are actually interchangeable, namely, the categories are equivalent \cite{Batchelor1980,Rogers1980,Rogers2007}. In particular, there is an isomorphism of sheaves that permits identifying a system of coordinates on an algebro-geometric supermanifold with a system of coordinates on the corresponding concrete supermanifold and vice versa.

The Frobenius theorem can be proven in a short and simple manner by separating its algebraic and analytic aspects, namely, by first proving that any involutive distribution can be locally generated by commuting vector fields $X_i$, and then showing that there exists a system of local coordinates $(x^a)$ which commonly straightens commuting linearly-independent vector fields, i.e.,~$X_i = \pa_{x^i}$. This proof for ordinary (purely even) manifolds is due to Lundell~\cite{Lundell:1992}. As Shander \cite{Shander:1980} proved, on supermanifolds there also exist straightening local coordinates for commuting linearly-independent vector fields (with the non-trivial homological condition $[X_i, X_i]=2X_i\circ X_i=0$ for odd vector fields). To prove the Frobenius theorem for involutive graded distributions, we will first show that an involutive distribution is locally generated by homogeneous vector fields in involution, and then show the existence of homogeneous straightening coordinates for those vector fields.

\bigskip\noindent
The paper is organized as follows. In Section~\ref{sec:Bursztyn}, we recall the main definitions and results by Bursztyn, Cueca, and Mehta. In Section~\ref{sec:homogeneity}, we introduce homogeneity supermanifolds, present several examples, and prove that $M=h_0(\MM)$ is a submanifold of $\MM$ and
\[
h_0:\MM\to M
\]
is a fiber bundle admitting local homogeneous coordinates. Although this result already appears in \cite{Jozwikowski:2016}, we include a complete proof for the reader's convenience and to make the paper self-contained.
In Section~\ref{sec:N_manifolds} we establish the equivalence of the various definitions of $\NN$-graded manifolds. Finally, in Section~\ref{sec:Frobenius} we compare the corresponding notions of graded distributions and prove the homogeneous Frobenius theorem.

\section[The sheaf-theoretic approach]{The sheaf-theoretic approach to $\NN$-graded manifolds}\label{sec:Bursztyn}
Let us recall the definitions used by Bursztyn, Cueca, and Mehta in the recent paper \cite{B.C.M2025}.
\begin{definition}
  Let $V = \bigoplus_{i=1}^n V_i$ be an $\NN$-graded vector space, where $n$ is a non-negative integer. An \emph{algebro-geometric $\NN$-graded manifold of degree $n$} is a ringed space $\MM =(M, \CM)$ consisting of a smooth manifold $M$ endowed with a sheaf of graded commutative algebras such that any point in $M$ admits a neighborhood $U$ with an isomorphism
  \begin{equation}\label{eq:sheaf_local_model}
    \restr{\CM}{U} \cong \Cinfty_U \otimes \operatorname{S}^\bullet V\, ,
  \end{equation}
  where $ \operatorname{S}^\bullet V$ denotes the graded symmetric algebra of $V$. The subsheaf of homogeneous functions of degree $l$ is given by
  $$\restr{\CM^l}{U}\cong \Cinfty_U \otimes \operatorname{S}^l V\, .$$
  The dimension of $\MM$ is $m_0|\cdots|m_n$, where $m_0 = \dim M$ and $m_i = \dim V_i$ for each $i=1, \ldots, n$. A chart
  \begin{equation}\label{eq:sheaf_chart}
    \left(U\subseteq M; x^\alpha, e_i^{\beta_i}\right)\, , \quad 1\leq \alpha \leq m_0\, , \quad 1\leq \beta_i \leq m_i\, , \quad 1\leq i\leq n
  \end{equation}
  of $\MM$ consists of a chart $(U; x^\alpha)$ of $M$ such that the condition \eqref{eq:sheaf_local_model} holds in $U$, and $\left(e_i^{\beta_i}\right)_{\beta_i=1}^{m_i}$ is a basis of $V_i$. Any homogeneous function over $U$ can be expressed as a sum of functions that are smooth in $x^i$ and polynomial in $e_i^{\beta_i}$.

  \mn A \emph{morphism of algebro-geometric $\NN$-graded manifolds of degree $n$} $\Psi\colon \MM \to \mathcal{N}$ is a morphism of ringed spaces, given by a pair $\Psi = (\psi, \psi^\sharp)$, where $\psi\colon M \to N$ is a smooth map and $\psi^\sharp \colon \Class_{\mathcal{N}} \to \psi_\ast \CM$ is a morphism of sheaves of algebras over $N$ which preserves the degrees. Algebro-geometric $\NN$-graded manifolds of degree $n$ with these morphisms form the category~$\mathsf{Man}_{\mathsf{alg}}^n$.
\end{definition}
\no The parity reversing functor is the endofunctor in the category of $\ZZ_2$-graded vector spaces defined by
$$\Pi\colon V \to \Pi V\, , \quad (\Pi V)_0 = V_1\, , \quad (\Pi V)_1 = V_0\, . $$
It relates the exterior and symmetric algebras of $V$ as follows:
$$\operatorname{S}^\bullet V = \bigwedge\nolimits^\bullet \Pi V\, , \quad  \bigwedge\nolimits^\bullet V = \operatorname{S}^\bullet \Pi V\, .$$
Every $\NN$-graded vector space is clearly also $\ZZ_2$-graded, since one can naturally define the $\ZZ_2$-degree as the $\NN$-degree modulo $2$.
Consequently, the local form \eqref{eq:sheaf_local_model} of the sheaf on an $\NN$-manifold can be rewritten as
 \begin{equation}\label{eq:sheaf_local_model_wedge}
    \restr{\CM}{U} \cong \Cinfty_U \otimes \bigwedge\nolimits^\bullet W\, ,
  \end{equation}
where $W\coloneqq \Pi V$.

\begin{definition}\label{def:distribution_Bursztyn}
  Let $\MM=(M, \CM)$ be an algebro-geometric $\NN$-graded manifold of degree $n$ and dimension $m_0|\cdots |m_n$. Given an open subset $U\subseteq M$, a vector field of degree $k$ on $\restr{\MM}{U}$ is a degree $k$ derivation $X$ of $\CM(U)$, i.e., an $\RR$-linear map $X\colon \CM(U)\to \CM(U)$ with the property that, for all $f, g\in \CM(U)$ with $f$ homogeneous,
  $\deg(X f) = \deg(f) + k$ and
  $$X(fg) = X(f) g + (-1)^{k \deg(f)} f X(g)\, .$$
  The sheaf of all vector fields on $\MM$ is denoted by $\mathscr{T}_\MM^\bullet$. The graded commutator of vector fields is defined by
  $$[X, Y] = XY - (-1)^{\deg(X)\, \deg(Y)} Y X$$
  for homogeneous vector fields $X$ and $Y$, and extended to any vector fields by linearity.

  A distribution on $\MM$ of rank $(d_0|\cdots |d_n)$ is a graded subsheaf $D\subseteq \mathscr{T}_\MM^\bullet$ of $\CM$-modules such that any point in $M$ admits an open neighborhood $U$ in which $D(U)$ is generated by linearly independent vector fields
  $$\{X_k^{i_k}\}\, , \quad  0\leq k \leq n\, , \quad 1\leq i_k \leq d_k$$
  which are homogeneous with $\deg(X_k^{i_k}) = -k$.
\end{definition}
One can equivalently view vector fields on $\MM$ as sections of tangent bundles \cite{Mehta2006}, and distributions as subbundles of $\T \MM$ which are locally freely generated by vector fields.

\bigskip\no
The main result by Bursztyn, Cueca, and Mehta \cite[Theorem 6.4]{B.C.M2025} is the following.
\begin{theorem}[Frobenius theorem for algebro-geometric $\NN$-graded manifolds]\label{theorem:Frobenius_Bursztyn}
    Let $D$ be an involutive distribution of rank $d_0|\cdots |d_n$ on an algebro-geometric $\NN$-graded manifold $\MM=(M, \CM)$ of degree $n$ and dimension $m_0|\cdots |m_n$. Then, around any point in $M$ there is a chart $(U; x^\alpha, e_j^{\beta_j})$ such that
  $$D(U) = \left\la \pa_{x^\alpha}, \pa_{e_j^{\beta_j}}\mid 1\leq \alpha \leq d_0\, ,\ 1\leq j\leq n\, ,\ 1\leq \beta_j \leq d_j \right\ran\, .$$
\end{theorem}
\no
Their proof relies on several sophisticated equivalences translating constructions from the category~$\mathsf{Man}_{\mathsf{alg}}^n$ to the category of admissible $n$-coalgebra bundles. This result is a special case of our Frobenius theorem for homogeneity supermanifolds, whose proof is simpler (see Section~\ref{sec:Frobenius}).

\section{Homogeneity supermanifolds}\label{sec:homogeneity}
In this section, we present homogeneity supermanifolds and their main properties.

\begin{definition}
\
\begin{itemize}
\item
A \emph{homogeneity action on a supermanifold $\MM$} is a smooth action $h\colon \RR \times \MM \to \MM$ of the multiplicative monoid $(\RR, \cdot)$ of real numbers. More precisely, $h$ is a morphism of supermanifolds defined by a collection of maps $h_t \colon \MM \to \MM,\, t\in \RR$ such that
\be\label{hs}h_t\circ h_s=h_{ts}\quad\text{and}\quad h_1=\id_\MM.\ee
In this case, we call $(\MM,h)$ an \emph{$h$-supermanifold}.
\item A morphism of two $h$-supermanifolds $(\MM_1, h^1)$ and $(\MM_2, h^2)$ is a morphism of supermanifolds $\Phi\colon \MM_1 \to \MM_2$ intertwining the actions $h_1$ and $h_2$,
\be\label{mor}
\Phi\circ h^1_t=h^2_t\circ\Phi\quad\text{for all}\quad t\in\R.
\ee
This gives rise to the obvious definition of the \emph{category of  $h$-supermanifolds}, $\mathsf{hMan}$.
\item The vector field
\be\label{wvf}\n=\frac{\pa}{\pa t}\,\Big|_{t=1}h_t\ee
we call the \emph{weight vector field} of the homogeneity action $h$. In other words,
$$\n(f)=\frac{\pa}{\pa t}\,\Big|_{t=1}(f\circ h_t) \, , \quad \forall f \in \Cinfty(\MM).$$
\item A (local) function $f$ on $\MM$ is called \emph{homogeneous of weight $w\in \NN$} if
\be\label{weight}f\circ h_t = t^w\cdot f\quad\text{for all}\quad t\in\R,\ee
or equivalently, if $\n (f) = w\cdot f$.
In this case, we write $\wa(f)=w$.
\item A (local) tensor field $K$ on $\MM$ is homogeneous of \emph{weight} $w\in \ZZ$ if
$$\liedv{\n} K = w\cdot K\,,$$
where $\liedv{}$ denotes the Lie derivative.
\item A submanifold $\MM_0$ in an $h$-supermanifold $\MM$ we call a \emph{homogeneity submanifold} (\emph{$h$-submanifold} for short) if $\MM_0$ is invariant with respect to the $\R$-action,
\be\label{sub} h_t(\MM_0)\subset \MM_0\quad\text{for all}\quad t\in\R.
\ee
\end{itemize}
\end{definition}
\begin{example}
The multiplication $h_t$ by real numbers $t\in\R$ in a vector superbundle $\zt:E\to M$ (i.e.,~$h_t$ acts by multiplication by scalars on the fibers and the identity on the base space $M$) is a homogeneity action.
In this case, the weight vector field is called the \emph{Euler vector field}, and the mappings $h_t$ are  called \emph{homotheties}.
\end{example}
\begin{example}
The product
$$E=E_1\ti_ME_2\ti_M\cdots\ti_ME_k$$
of vector superbundles over a supermanifold $M$ is canonically an $h$-supermanifold with the action
$$h_t(v_1+v_2+\cdots+v_k)=t^1v_1+t^2v_2+\cdots+t^kv_k.$$
Such $h$-manifolds are called \emph{split $h$-manifolds}.
\end{example}
\no In what follows, we will deal exclusively with supermanifolds and generally skip the prefix `super'.

\mn The following is the super-analog of the well-known result on smooth retractions of purely even manifolds (see \cite[Theorem 1.13]{Kolar:1993}). It is of independent interest, but we will need the following corollary in the sequel.
\begin{proposition}
Let $\phi:\mathcal M\longrightarrow \mathcal M$ be a smooth idempotent endomorphism of a smooth supermanifold, i.e., $\phi^2=\phi$. Then $\phi$ has locally constant rank.
\end{proposition}

\begin{proof}
Let
$$
\phi^\ast:\mathcal O_{\mathcal M}\longrightarrow \mathcal O_{\mathcal M}
$$
be the corresponding morphism of sheaves. The condition $\phi^2=\phi$ is equivalent to $(\phi^\ast)^2=\phi^\ast $.
Hence, $\phi^\ast$ is a projector. The induced morphism on K\"ahler differentials,
$$D\coloneqq (\sT \phi)^\ast:\Omega^1_{\mathcal M}\longrightarrow \Omega^1_{\mathcal M},$$
is again idempotent,  $D^2=D$, therefore
$$\Omega^1_{\mathcal M}=\operatorname{Im}\bl D\br\oplus\ker\bl D\br.
$$
Thus $\operatorname{Im}(D)$ is a direct summand of the locally free $\mathcal O_{\mathcal M}$-module
$\Omega^1_{\mathcal M}$. Consequently, it is locally free.
Hence, locally,
$$\operatorname{Im}(D)\simeq\mathcal O_{\mathcal M}^{m|n},
$$
for some integers $m,n$. Since the rank of a locally free module is locally constant, the rank of $\phi$ is locally constant.

\end{proof}

\begin{corollary}\label{c1}
Every smooth idempotent endomorphism
$$\phi:\mathcal M\to\mathcal M,\qquad \phi^2=\phi,$$
is locally of constant rank. Therefore, by the constant rank
theorem, $\phi$ locally factors as
$$\mathcal M\xrightarrow{\ \rho\ } M \xrightarrow{\ \iota\ }\mathcal M,
$$
where $\rho$ is a submersion, $\iota$ is an embedding, and
$$\phi=\iota\circ\rho,\qquad\rho\circ\iota=\operatorname{id}_M.
$$
Thus, the image $M$ of $\phi$ is a smooth retract of $\mathcal M$.

\mn In particular, for every $x_0\in|M|$ there exist local coordinates $(x^1,\dots,x^r,y^1,\dots,y^n)$ on $\mathcal M$ centered at $x_0$, such that $\phi$ is locally isomorphic to the projection
$$(x^1,\dots,x^m,y^1,\dots,y^n)\longmapsto(x^1,\dots,x^m,0,\dots,0).$$
\end{corollary}

\mn The following theorem is a super-analog of the main result in \cite{Grabowski:2012}, proved there for the standard (purely even) manifolds. The super-version already appeared in \cite{Jozwikowski:2016}, and the proof there makes substantial use of \cite{Grabowski:2012}. To keep our paper self-contained, we present here a full proof working for supermanifolds.

\begin{theorem}\label{theorem:homogeneous_trivialization}
Let $(\MM,h)$ be an $h$-manifold. Then $h_0:\MM\to\MM$ is a smooth retract onto a submanifold $M$ of $\MM$.
Moreover,
$$h_0:\MM\to M$$
is a fiber bundle with the typical fiber $F=\R^{\za|\zb}$, for some $\za,\zb\in\NN$, and for every $x_0\in|M|$ there exists a local trivialization
$$h_0^{-1}(U)\simeq U\ti\R^{\za|\zb},$$
with coordinates $(x^1,\dots,x^m;u^1,\dots,u^n)$ on $U\ti\R^{\za|\zb}$, where $(x^a)$ are local coordinates on $M$  centered at $x_0$, and $(u^i)$ are the standard global coordinates on $\R^{\za|\zb}$, such that
\be\label{hoco} (x^1,\dots,x^m,u^1,\dots,u^n)\circ h_t=(x^1,\dots,x^m,t^{w_1}u^1,\dots,t^{w_n}u^n),
\ee
for some positive integers $w_i$ and all $t\in\R$. If, additionally, the diffeomorphism $h_{-1}$ acts as the parity operator, then $M$ is purely even, and the weight determines the parity: the coordinate $u^i$ is even (resp., odd) if $w_i$ is even (resp., odd).
\end{theorem}
\begin{proof}
As $h_0$ is a smooth idempotent endomorphism of $\MM$, we get from Corollary \ref{c1} that the image of $h_0$ is a smooth retract $M$ of $\MM$, and for every $x_0\in|M|$ there exist local coordinates $(x^1,\dots,x^m,y^1,\dots,y^n)$ on $\mathcal M$ centered at $x_0$, such that $\phi$ is locally isomorphic to the projection
$$(x^1,\dots,x^m,y^1,\dots,y^n)\circ h_0=(x^1,\dots,x^m,0,\dots,0).$$
Hence, we can view $(x^a)$ as local coordinates on $M$, and $(y^i)$ as local fiber coordinates in the typical local fiber $V$. Since $h_t\circ h_0=h_0=h_0\circ h_t$, the coordinates $(x^a)$ and the local fibers are respected by all $h_t$, $t\in \R$. Moreover, for $q\in |M|$, the endomorphism $H_t(q)=\bl\sT_qh_t\br^*$ acts on the vector superspace $\sT^*_{q}\MM$, and $h_t\circ h_s=h_{ts}$ implies
\be\label{Hs}H_t\circ H_s=H_{ts},\ee
for all $s,t\in\R$. Consider the Taylor expansion $A_t=\sum_{k=0}^\infty t^kP_k$ of $H_t$ at $t=0$.
From (\ref{Hs}), we get $A_tA_s=A_{ts}$, so
$$\sum_{k,l}t^ks^lP_kP_l=\sum_kt^ks^kP_k,$$
and $P_kP_l=\zd^k_lP_k$, where $\zd^k_l$ is the Kronecker symbol. In other words, $\bl P_k(q)\br$ is a decomposition of the identity on $\sT^*_q\MM$ into linear projections,
$$P_k(q):\sT^*_q\MM\to E_k(q)\subset\sT^*_q\MM,$$
and $E_k(q)\cap E_l(q)=\{0\}$ for $k\ne l$. Since $P_k(q)$ smoothly depends on $q\in M$, the dimension $d_k(q)$ of $E_k(q)$ is locally constant. But the number of trivially intersecting nontrivial vector subspaces of a finite-dimensional vector superspace is not bigger than its dimension, so there exists the smallest $r\in\NN$ such that
$$A_t=\sum_{k=0}^r t^kP_k,$$
for $t\in\R$. In particular, $d_r>0$.

\mn Since the polynomial $A_t$ is the Taylor expansion of $H_t$, we have $H_t=A_t+B_t$, where the smooth curve $\R\ni t\to B_t(q)\in\End(\sT^*_q\MM)$ is flat at $t=0$, i.e., for any $n=1,2,\dots$ there exists a $t_n>0$ such that
\be\label{flat} \nm{B_t}<t^n\quad\text{for}\quad 0<t<t_n,\ee
where $\nm{\cdot}$ is the operator norm in $\End(\sT^*_q\MM)$, associated with a scalar product. It is easy to see that if $B_t$ is not trivially 0 for $t$ close to 0, then we can choose a sequence $t_n\to 0$ such that $t_n$ is the smallest positive number with property (\ref{flat}), so $\nm{B_{t_n}}=t^n_n$. Hence,
$$t_n^n=\nm{B_{t_n}}=\nm{B_2B_{t_n/2}}\le\nm{B_2}\cdot\nm{B_{t_n/2}}<\nm{B_2}\cdot(t_n/2)^n,$$
and we get a contradiction when $n$ tends to infinity. Therefore,
$$
H_t=\sum_{k=0}^r t^kP_k,
$$
where $P_k$ is a projection on the vector subbundle $E_k$ of rank $d_k$. But $H_t$ is an isomorphism for $t\ne 0$, so
\be\label{V}
\sT^*_q\MM=\bigoplus_{k=0}^rE_k(q),
\ee
and $d_k=\dim E_k(q)$ is locally constant. Of course, $E_0(q)\simeq\sT^*_qM$.

\mn We can assume now that the fiber coordinates $(y^i)$ are chosen such that
\be\label{ccoorr}\xd y^i(q)\in E_{w_i}(q)\ee
for $q$ in a neighbourhood of $x_0$ in $|M|$, and $w_i \in \{1, \ldots, r\}$. From
$$(y^i\circ h_t)\circ h_s=y^i\circ h_{ts},$$
we get that the functions $y^i_j$ in the Taylor expansions,
$$y^i\circ h_t=ty^i_1+\cdots+t^ry^i_r+o(t^r),
$$
are homogeneous of weight $j$, namely, $y^i_j\circ h_t=t^jy^i_j$, and that
(cf.~(\ref{ccoorr}))
$$\xd y^i_{w_i}(q)=\xd y^i(q).$$
The fact that $\bl\xd x^a(q),\xd y^i(q)\br$ span $\sT^*_q\MM$ means that $\bl x^a,z^i=y^i_{w_i}\br$ is a system of homogeneous coordinates in a neighbourhood of $x_0$ in $\MM$,
\be\label{hoko}x^a\circ h_t=x^a,\quad z^i\circ h_t=t^{w_i}z^i,\ee
for sufficiently small $t\in\R$.

\mn Since odd coordinates are global in their nature, let us concentrate now on even coordinates, which are only locally defined.
Consider a neighbourhood $W_0\subseteq M=|\MM|$ of $x_0\in M$ with even coordinates from (\ref{hoko}), say
$$\bl x^1,\dots,x^s;z^1,\dots, z^\za\br.$$
We can split it as the Cartesian product $W_0=U_0\ti V_0$ such that $(x^1,\dots,x^s)$ and $(z^1,\dots, z^\za)$ are coordinates on $U_0$ and on $V_0$, respectively.
As all $h_t$ are morphisms of supermanifolds, they respect the body $|\MM|$, and define on the  purely even manifold $|\MM|$ the reduced $h$-action $h'$. This action is trivial on $U_0$ and
$$\bl z^1,\dots, z^\za\br\circ h_t=\bl t^{w_1}z^1,\dots, t^{w_\za}z^\za\br,$$
for sufficiently small $t$. The open submanifold $F_0=(h'_0)^{-1}(U_0)$ in $|\MM|$ is invariant with respect to the $h'$-action, and for any $q\in F_0$ there exists a sufficiently large $k(q)\in\NN$ such that $h_{1/k(q)}(q)\in W_0$.
Let us define an $h$-action $h^0$ on $\R^\za$ with global coordinates $(u^1,\dots,u^\za)$ by
$$\bl u^1,\dots,u^\za\br\circ h^0_t=\bl t^{w_1}u^1,\dots,t^{w_\za}u^\za\br.$$
It is easy to see now that the map,
$$\Phi:F_0\to U_0\ti\R^\za,\quad q\mapsto \Bl t^{w_1k(q)}u^1\bl h_{1/k(q)}(q)\br,\dots,t^{w_\za k(q)}u^\za\bl h_{1/k(q)}(q)\br\Br,
$$
is a diffeomorphism intertwining the $h$-actions. Now, we add the homogeneous odd coordinates to the picture, (\ref{hoco}) follows, and the final statement is obvious.

\end{proof}
\begin{corollary}
The weight vector field in the homogeneous coordinates (\ref{hoco}) reads
\be\label{wvf2} \n=\sum_{i=1}^nw_i\cdot u^i\pa_{u^i},
\ee
thus it is even, complete, and its flow reads $\Exp(s\n)=h_{e^s}$.
\end{corollary}
\no The biggest $w_i$ in (\ref{hoco}), we call the \emph{degree} of the $h$-action (or the degree of the $h$-manifold $\MM$).

\begin{proposition}
  Let $(\MM,h)$ be an $h$-manifold with homogeneous coordinates $(x^a, u^i)$ defined on a local trivialization $h_0^{-1}(U)$ as in Theorem~\ref{theorem:homogeneous_trivialization}. If $f$ is a (local) homogeneous function of weight $k\in \NN$ on $\MM$, then it is polynomial in the variables $u^i$, with coefficients being (local) smooth functions in the variables $x^a$.
\end{proposition}

\begin{proof}
  We shall show the result by induction on $k$. Clearly, $0$-homogeneous functions are independent of the variables $u^i$. Assume that $f$ has weight  $k>0$. Then, $\liedv{\n} f = k\cdot f$, and thus
  $$\liedv{\n} \liedv{\pa_{u^i}} f = \liedv{[\n, \pa_{u^i}]} f + k\cdot \liedv{\pa_{u^i}} f = (k-w^i) \frac{\partial f}{\partial u^i}\, ,\quad 1\leq i\leq n\, ,$$
  which means that $\frac{\partial f}{\partial u^i}$ has weight $k-w^i<k$. By the induction hypothesis, $\frac{\partial f}{\partial u^i}$ is a polynomial in the variables $u^i$, with coefficients being (local) smooth functions in the variables $x^a$. Because this holds for all partial derivatives with respect to the variables $u^i$, it also holds for $f$.
\end{proof}

\begin{corollary}
  Transition functions between systems of homogeneous bundle coordinates on an $h$-manifold are polynomials in the homogeneous fiber coordinates. Moreover, all systems of homogeneous bundle coordinates on a connected component have the same weights up to permutations.
\end{corollary}

\section[The homogeneity approach]{The homogeneity approach to $\NN$-graded manifolds}\label{sec:N_manifolds}
Let $\MM =(M, \CM)$ be an algebro-geometric $\NN$-graded manifold with a chart \eqref{eq:sheaf_chart}.
Note that $\MM$ is naturally equipped with a weight vector field $\n$ locally given by
$$\n = \sum_{j=1}^{n} \sum_{\beta_j=1}^{m_j} j\, e_j^{\beta_j}\, \pa_{e_j^{\beta_j}}\, ,$$
which corresponds to declaring $(x^\alpha)$ to have weight zero, and $e_j^{\beta_j}$ to have as weight the degree of $V_j$ in the graded vector space $V$. The flow of $\n$ defines an action of the monoid $(\RR, \cdot)$ on $\MM$ given by
\begin{equation}\label{eq:homog_structure_from_sheaf}
  h_t\left(x^\alpha, e_i^{\beta_i}\right) = \left(x^\alpha, t^i e_i^{\beta_i}\right)\, .
\end{equation}
In particular, $h_0$ can be identified with the projection onto the body $M$. Additionally, note that $h_{-1}$ acts as the parity operator.

\mn Motivated by the above observations, we propose the following definition.
\begin{definition}
An \emph{N-manifold of degree $n$} is an $h$-manifold $\MM$ such that the $h$-action $h_t$ has degree $n$, and $h_{-1}$ acts as the parity operator. In particular, the base manifold $M=h_0(\MM)$ is purely even. N-manifolds of degree $n$ form the category $\mathsf{NMan}^n$, a full subcategory of the category $\mathsf{hMan}$ of $h$-manifolds.
\end{definition}
\begin{remark} In our approach, the question of `respecting weights' is not fundamental. Morphisms of $h$-manifolds respect weights by definition. Note, however, that the 0-function is of any weight, so we accept morphisms which kill some of the homogeneous functions; therefore, there are morphisms into $h$-manifolds of bigger degree, as well as morphisms into $h$-manifolds of smaller degree.
Note also that our definition of N-manifolds was already suggested by {\v{S}}evera \cite{Severa:2005}, however, without any proofs or details.
It is also worth mentioning that the equivalence between Roytenberg's and  {\v{S}}evera's definitions is an immediate consequence of Theorem~\ref{theorem:homogeneous_trivialization}.
\end{remark}

\begin{theorem}
  The categories $\mathsf{Man}_{\mathsf{alg}}^n$ and $\mathsf{NMan}^n$ are equivalent. Moreover, on each $\NN$-graded manifold $\MM$, the degrees of homogeneous functions defined by the homogeneity action $h$ and by the grading of the sheaf $\CM$ coincide.
\end{theorem}

\begin{proof}
  As we have already seen, given an algebro-geometric $\NN$-graded manifold, one can construct the homogeneity action \eqref{eq:homog_structure_from_sheaf} such that $h_0$ is the projection onto the body and $h_{-1}$ is the parity operator.

  Conversely, given a homogeneous $\NN$-graded supermanifold $(\MM, h)$, one can define an associated algebro-geometric $\NN$-graded manifold as follows.
  One can always find an atlas on $\MM$ with homogeneous coordinates $(z^a)$ on $\MM$ with non-negative integer degrees $(w_a)$. In the algebro-geometric characterization, $\MM = M$ is an ordinary manifold $M$ with a sheaf $\CM$ locally isomorphic to $\Cinfty(U)\otimes \bigwedge\nolimits^\bullet(V)$ for some vector space $V$.

  \mn We can rewrite the coordinates as $(z^a)=(x^\alpha, y^\mu)$, where $x^\alpha$ and $y^\mu$ have zero and non-zero weights, respectively; and express their set of weights as $\Gamma=\{0, w_1, \ldots, w_k\}$, with $0<w_1<\ldots <w_k \eqqcolon n$, and $k,n$ being non-negative integers.
  The assumption that $h_0$ is the projection onto the body means that $(x^\alpha)$ are precisely the coordinates on the body. On the other hand, the coordinates $(y^\mu)$ define a basis of $V$.

  \mn Let $V_i$ be the subspace of $V$ generated by coordinates $(y^\mu)$ with weight $w_i$, or $V_i=\{0\}$ if there are no coordinates of that weight. This makes  $V$ into a graded vector space:
  $$V= \bigoplus_{i=1}^n V_i\, .$$
  The assumption that $h_{-1}$ acts as the parity operator implies that the subspace $V_i$ has the $\ZZ_2$-degree
  $$w_i\ (\mod 2) = i\ (\mod 2) \, .$$
  In plain words, $V_i$ is even (resp.~odd) if $i$ is even (resp.~odd).

  \mn By construction, it is clear that a morphism of supermanifolds $\Phi \colon \MM \to \NN$ preserves the $\NN$-grading of the sheaves if and only if it intertwines the corresponding homogeneity actions.

\end{proof}

\section{Homogeneous distributions and the Frobenius theorem}\label{sec:Frobenius}
Let $(\MM, h)$ be an $h$-manifold with weight vector field $\n$.
A \emph{homogeneous distribution} is a vector subbundle $D\subseteq \T \MM$ such that
$$\T h_t (D) \subseteq D\, , \quad \forall\, t\in\RR\, ,$$
or, equivalently, such that
$$ \liedv{\n} X = [\n,X]\in D$$
for any vector field $X$ from $D$, if we identify $D$ with the corresponding locally free module of vector fields.

 A vector field $X$ is homogeneous of weight $w$ if and only if
$$ \liedv{\n} X = [\n,X]=w \cdot X.$$
In particular, for homogeneous bundle coordinates (\ref{hoco}), $\pa_{x^a}$ has weight 0, and $\pa_{u^i}$ has weight $-w_i$.

\begin{theorem}
  Let $(\MM, h)$ be an $h$-supermanifold.
  A distribution $D\subseteq \T \MM$ of rank $k$ is homogeneous if and only if, for every $x_0 \in |M|$, there exists a local trivialization $h_0^{-1}(U)\simeq U \times \RR^{\alpha |\beta}$ of the fiber bundle $h_0:\MM\to M$,
  and $k$ homogeneous vector fields $X_1, \ldots, X_r, Y_1, \ldots, Y_{k-r}$ which generate $D$ on $h_0^{-1}(U)$, where $X_i$ have zero weight. Moreover, $D_0 = \T h_0(D)\subseteq \T M$ is a distribution on $M = h_0(\MM)$ generated by $X_1, \ldots, X_r$ on $U$.
\end{theorem}

\begin{proof}
  Obviously, linearly independent homogeneous vector fields generate a homogeneous distribution. Conversely, given a homogeneous distribution $D\subseteq \T \MM$, for each point $x_0 \in |M|$, $D_{x_0}$ is an $\NN$-graded vector subspace of $\T_{x_0}\MM$. By Theorem~\ref{theorem:homogeneous_trivialization}, there exists a system of homogeneous bundle coordinates $(x^1,\dots,x^m;u^1,\dots,u^n)$ on a local trivialization $h_0^{-1}(U)\simeq U \times \RR^{\alpha |\beta}$, where $U$ is a neighborhood of $x_0$ in $|M|$.  We may assume that $\restr{\pa_{x^1}}{x_0}, \ldots, \restr{\pa_{x^r}}{x_0}, \restr{\pa_{u^1}}{x_0}, \restr{\pa_{u^{k-r}}}{x_0}$ is a graded basis of $D_{x_0}$ (changes of bases preserving the grading have associated linear changes of homogeneous coordinates). Additionally, shrinking $U$ if necessary, we may assume that $D$ is generated by  $X_1, \ldots, X_r, Y_1, \ldots, Y_{k-r}$ on $h_0^{-1}(U)$, with
  $\restr{X_i}{x_0} = \pa_{x^i}$ and $\restr{Y_a}{x_0} = \partial_{u^a}$ for $1\leq i \leq r$ and $1\leq a \leq k-r$. Because sections of $D$ form a $\Cinfty(\MM)$-module, we can additionally assume that
  \begin{equation}\label{eq:generators_distribution}
  \begin{aligned}
    & X_i = \pa_{x^i} + \sum_{j=r+1}^m f_i^j \pa_{x^j} + \sum_{b=k-r+1}^n g_i^b \pa_{u^b}\, , \quad 1\leq i\leq r\, ,\\
    & Y_a = \pa_{u^a} + \sum_{j=r+1}^m A_a^j \pa_{x^j} + \sum_{b=k-r+1}^n B_a^b \pa_{u^b}\, , \quad 1\leq a\leq k-r\, .
  \end{aligned}
  \end{equation}
  for some (super)functions $f_i^j,g_i^b,A_a^j,B_a^b$. Thus, for every $t\in \RR$,
  \begin{align*}
    & [\n,X_i] = 0+\sum_{j=r+1}^m \n(f_i^j)\, \pa_{x^j} + \sum_{b=k-r+1}^n  \left(\n(g_i^b)-w_bg_i^b\right)\, \pa_{u^b}\, , \quad 1\leq i\leq r\, ,\\
    & [\n,Y_a] = -w_a\, \pa_{u^a} + \sum_{j=r+1}^m  \n(A_a^j)\, \pa_{x^j} + \sum_{b=k-r+1}^n \left(\n(B_a^b)-w_bB_a^b\right)\,\pa_{u^b}\, , \quad 1\leq a\leq k-r\, .
  \end{align*}
  On the other hand, $[\n,X_i]$ and $[\n,Y_a]$ are linear combinations of $X_1, \ldots, X_r, Y_1, \ldots, Y_{k-r}$, so
  $X_i$ has weight 0, and $Y_a$ has weight $-w_a$, which proves the theorem.

\end{proof}
\no Note that a version of the above theorem has been proven in \cite{Grabowska:2025} for more general homogeneity supermanifolds.
\begin{corollary}
  Let $(\MM, h)$ be an $h$-manifold. A distribution $D\subseteq \T \MM$ of rank $k$ is homogeneous and involutive  if and only if for every $x_0 \in |M|$ there exists a local trivialization $h_0^{-1}(U)\simeq U \times \RR^{\alpha |\beta}$ of the fiber bundle $h_0:\MM\to M$, and $k$ homogeneous pairwise commuting vector fields $Z_1, \ldots, Z_k$, generating $D$ on $h_0^{-1}(U)$.
\end{corollary}

\begin{proof}
  Consider the generators \eqref{eq:generators_distribution} of $D$ on $h_0^{-1}(U)$. Note that $[X_i, X_j],\ [X_i, Y_a]$ and $[Y_a, Y_b]$ are linear combinations of $\pa_{x^{r+1}}, \ldots, \pa_{x^m}, \pa_{u^{k-r+1}}, \ldots, \pa_{u^n}$. On the other hand, since $D$ is involutive, they are linear combinations of $X_1, \ldots, X_r, Y_1, \ldots, Y_{k-r}$. As a result, they all commute.

\end{proof}
\no Note that for odd vector fields $Z_i$, the vanishing of $[Z_i,Z_i]=0$ is a nontrivial condition stating that $Z_i$ is \emph{homological}. Below we prove the homogeneous super-version of the well-known result from the standard differential geometry saying that linearly independent and commuting vector fields can be locally expressed as coordinate vector fields.
\begin{proposition}
  Let $X_1, \ldots, X_k$ be pairwise commuting homogeneous vector fields on an $h$-manifold $(\MM, h)$. If $X_i$ are linearly independent at a point $x_0\in |M|$, then there exists a system of bundle homogeneous coordinates $(x^i)$ centered at $x_0$ such that $X_i = \pa_{x^i},\, 1\leq i\leq k$.
\end{proposition}

\begin{proof}
  The result can be proven, \textit{mutatis mutandis}, like Proposition V.6.1 in \cite{Tuynman2005}. Let $(y^1, \ldots, y^n)$ be a system of homogeneous coordinates centered at $x_0$. Up to a linear change of homogeneous coordinates, we may assume that $X_i(x_0) = \pa_{y^i}$. The desired change of homogeneous coordinates is given by $(x^1, \ldots, x^k,y^{k+1},\ldots,y^n)= \psi^{-1}(y^1, \ldots, y^n)$, where the map $\psi$ is defined by
  $$\psi(x^1, \ldots,x^k,y^{k+1},\ldots,y^n) = \phi^{X_1} \Big(x^1, \phi^{X_2}\big(x^2, \ldots, \phi^{X_k}(x^k, 0, \ldots, 0, y^{k+1}, \ldots, y^n)\big)\Big)\, ,$$
  with $\phi^{X_i}$ being the commuting flows of $X_i$, $i=1,\dots,k$ (flows with an odd `time' in the case of odd vector fields which, due to commutativity, are integrable). The existence of flows follows, e.g., from the `straightening theorem' \cite{Shander:1980}. In these coordinates, the homogeneous vector fields read $X_i = \pa_{x^i}$, so the coordinates $(x^1, \ldots, x^k,y^{k+1},\ldots,y^n)$ are homogeneous.

\end{proof}
\no Combining the results above, we obtain the homogeneous Frobenius theorem.
\begin{theorem}[Frobenius theorem for homogeneity supermanifolds]\label{theorem:Frobenius}
  A distribution $D\subseteq \T \MM$ of rank $k$ on an $h$-supermanifold $(\MM, h)$ is homogeneous and involutive if and only if, for every $x_0 \in |M|$, there exists a system of bundle homogeneous coordinates $(x^i)$ centered at $x_0$ such that
  $$D = \left\langle \pa_{x^1}, \ldots, \pa_{x^k}\right\rangle\, .$$
\end{theorem}
\no We finish by relating this result with the Frobenius theorem by Bursztyn, Cueca and Mehta \cite{B.C.M2025}.
Taking into account the relationships discussed in the previous section, the degrees of homogeneous vector fields defined by the homogeneity action and by the grading of the sheaf coincide. More specifically, a vector field $X$ on $\MM$ has $\deg(X) = -k$, in the sense of Definition~\ref{def:distribution_Bursztyn}, if and only if $\liedv{\n} X = -k\cdot X$.
This way, we come to the following conclusion.

\begin{proposition}
  There is a one-to-one correspondence between homogeneous distributions on a homogeneous $\NN$-graded supermanifold and distributions on an algebro-geometric $\NN$-graded manifold.
\end{proposition}
\no
Consequently, the main theorem \cite[Theorem 2.3]{B.C.M2025} is a direct corollary of Theorem~\ref{theorem:Frobenius}.






\bibliographystyle{plain}
\bibliography{biblio}

\vskip.5cm
\noindent Janusz Grabowski\\\emph{Institute of Mathematics, Polish Academy of Sciences}\\{\small ul. \'Sniadeckich 8, 00-656 Warszawa, Poland}\\
{\tt jagrab@impan.pl}\\  https://orcid.org/0000-0001-8715-2370
\\

\noindent Asier L\'opez-Gord\'on\\\emph{Institute of Mathematics, Polish Academy of Sciences}\\{\small ul. \'Sniadeckich 8, 00-656 Warszawa, Poland}\\
{\tt alopez-gordon@impan.pl}\\  https://orcid.org/0000-0002-9620-9647
\\

\end{document}